\documentclass[11pt,leqno]{article}
\usepackage{enumerate}
\usepackage{amsmath}
\usepackage{amssymb}
\usepackage[english]{babel} 
\usepackage{microtype}

\newtheorem{defeng}{Definition}[section]
\newtheorem{theorem}[defeng]{Theorem}

\newtheorem{lemma}[defeng]{Lemma}

\newtheorem{conjecture}[defeng]{Conjecture}

\newcounter{claim}

\newenvironment{proof}[1][]%
 {\noindent {\setcounter{claim}{0}\sc proof ---
   }{#1}{}}{\hfill$\Box$\vspace{2ex}} 

\newenvironment{claim}[1][]%
{\refstepcounter{claim}\vspace{1ex}\noindent{\it{\textit{Claim }\arabic{claim}{#1}{: }}}\it}{\vspace{1ex}}

\newenvironment{proofclaim}[1][]%
	{\noindent {}{#1}{}}{ This proves Claim~\arabic{claim}.\vspace{1ex}}

\newcommand{\ov}{\overline}

\title{The Chen-Chv\'atal conjecture for metric spaces induced by distance-hereditary graphs}
\author{Pierre Aboulker\thanks{E-mail: pierreaboulker@gmail.com} {} and Rohan Kapadia\thanks{E-mail: rohan.f.kapadia@gmail.com} \\ Concordia University, Montr\'eal, Qu\'ebec, Canada}

\begin{document}

\maketitle

\begin{abstract}
A classical theorem of Euclidean geometry asserts that any noncollinear set of $n$ points in the plane determines at least 	$n$ distinct lines. Chen and Chv\'atal conjectured a generalization of this result to arbitrary finite metric spaces, with a particular definition of lines in a metric space. 
We prove it for metric spaces induced by connected distance-hereditary graphs -- a graph $G$ is called distance-hereditary if the distance between two vertices $u$ and $v$ in any connected induced subgraph $H$ of $G$ is equal to the distance between $u$ and $v$ in $G$.
\end{abstract}

\section{Introduction}

It is well-known that 
\begin{itemize}
\item every non-collinear set of $n$ points in the Euclidean plane determines at least $n$ distinct lines.
\end{itemize}

As noted by Erd\H{os} \cite{E43}, this fact is a corollary of the Sylvester-Gallai theorem (which asserts that, for every non-collinear set $S$ of $n$ points in the plane, some line goes through precisely two points of $S$).
Coxeter \cite{Cox} gave a proof of the Sylvester-Gallai theorem using \textit{ordered geometry}: that is, without using notions of measurement of distances or measurement of angles, but instead employing the notion of \textit{betweenness}. A point $b$ is said to \textit{lie between} points $a$ and $c$ if $b$ is the interior point of the line segment linking $a$ and $c$.
We write $[abc]$ for the statement that $b$ lies between $a$ and $c$. 
In this notation, a \textit{line} $\ov{uv}$ is defined (for any two distinct points $u$ and $v$) as
\begin{equation} \label{definitionofline}
\ov{uv}=\{u,v\} \cup \{ p:[puv]\vee [upv]\vee[uvp]\}.
\end{equation}

Betweenness in metric spaces was first studied by Menger \cite{menger}. 
In a metric space $(V,d)$, we define
\[
[abc]  \Leftrightarrow d(a,b) + d(b,c) = d(a,c).
\]

Hence, in any metric space $(V,d)$, we can define the line $\ov{uv}$ induced by two points $u$ and $v$ as in (\ref{definitionofline}).
With this definition of lines in metric spaces, Chen and Chv\'{a}tal \cite{CC1} proposed the following beautiful conjecture.
A line of a metric space $(V,d)$ is \textit{universal} if it contains all points of $V$.

\begin{conjecture} \label{conjChvatal}
 Every metric space on $n$ points, where $n\ge 2$, either has at least $n$ distinct lines or has a universal line.
\end{conjecture}

This conjecture is wide open today, but a positive answer would reveal an iceberg of which the original theorem of Euclidean geometry is a tip. 
Some partial results have been obtained: Bondy, Chen, Chv\'atal, Chiniforooshan, Miao and the first author~\cite{hypergraph} proved that any metric space on $n$ points $(n \geq 2)$ has at least $(2-o(1)) \log_2 n$ distinct lines.
Actually, this result holds in the more general framework of lines in $3$-uniform hypergraphs; more on this subject can be found in \cite{BBCCCCFZ:linesinhypergraphs}.


It suffices to prove Conjecture~\ref{conjChvatal} for metric spaces with integral distances\footnote{This was pointed out to us by Xiaomin Chen and Va\v sek Chv\'atal.}.
To see this, we note that the set of lines in a finite metric space $(V, d)$ depends only on whether or not $[uvw]$ holds for each triple $(u, v, w)$ of distinct points of $V$; in other words, it depends on the following system of linear equations and inequalities being satisfied with $x_{uv} = d(u, v)$ for all distinct $u, v$ in $V$:
\begin{align*}
 x_{uv} + x_{vw} - x_{uw} = 0, & \text{ if } [uvw] \text{ holds}, \\
 x_{uv} + x_{vw} - x_{uw} > 0, & \text{ if } [uvw] \text{ does not hold}.
\end{align*}
If $(V, d')$ is another metric space on the same ground set and the above system holds with $x_{uv} = d'(u, v)$ for all $u, v$ in $V$, then $(V, d')$ has the same set of lines as $(V, d)$.
This system has a non-negative solution given by the distances $d(u, v)$ of the metric space $(V, d)$. Since it is homogenous and has integer coefficients, it also admits a non-negative integral solution, which gives us a metric space $(V, d')$ with integral distances that has the same set of lines as $(V, d)$.

This observation motivates looking at two particular types of metric spaces. First, for a positive integer $k$, we define a \textit{k-metric space} to be a metric space in which all distances are integral and are at most $k$.
Chv\'{a}tal \cite{Chvatal} proved that every $2$-metric space on $n$ points $(n \ge 2)$ either has at least $n$ distinct lines or has a universal line.

A second type of metric space with integral distances arises from graphs.
Any finite connected graph induces a metric space on its vertex set, where the distance between two points $u$ and $v$ is defined as the smallest number of edges in a path linking $u$ and $v$.
Conjecture~\ref{conjChvatal} has been proved for metric spaces induced by chordal graphs \cite{BBC2}; these are the graphs with no induced cycles of length four or more.

Metric spaces induced by graphs can behave strangely when we take induced subgraphs. Indeed, let $G$ be a graph and $H$ a connected induced subgraph of $G$: then the metric space induced by $H$ may not be a subspace of the metric space induced by $G$.
This is because the distance between two vertices may be greater in $H$ than in $G$, if none of the shortest paths joining them in $G$ are contained in $H$.
The \emph{distance-hereditary} graphs are precisely the class of graphs in which this does not happen.
We denote the distance between two vertices $u$ and $v$ in a graph $G$ by $d_G(u, v)$. Then we say that $G$ is \emph{distance-hereditary} if for any connected induced subgraph $H$ of $G$ and for any pair of vertices $x, y$ in $H$, we have $d_H(x, y) = d_G(x, y)$.
This class of graphs is particularly interesting from the point of view of the Chen-Chv\'atal conjecture because of this property; the metric space induced by any connected induced subgraph of $G$ is actually a subspace of the metric space induced by $G$ itself.
The study of distance-hereditary graphs was initiated by Howorka \cite{Howorka} who characterized them in several ways.

In this paper, we prove the conjecture for metric spaces induced by these graphs:

\begin{theorem} \label{main}
Every metric space induced by a connected distance-hereditary graph on $n$ vertices, where $n \geq 2$, either has at least $n$ distinct lines or has a universal line.
\end{theorem}

We hope that the structural techniques we use will shed light on a solution to the conjecture for all metric spaces induced by graphs.

\subsection*{Notation and preliminaries}

All graphs in this paper are simple and undirected. 
Let $G$ be a graph.
For a subset $S$ of $V(G)$, we let $G[S]$ denote the subgraph of $G$ induced by $S$, and $G - S = G[V(G)- S]$. 
Let $x$ be a vertex of $G$. 
We denote by $S_i(x)$  the set of vertices at distance $i$ from $x$.
We denote by $N(x)$ the \emph{neighbourhood} of $x$; that is, the set of vertices adjacent to $x$. We call two vertices $x$ and $y$   \emph{twins} if $N(x) - \{y\} = N(y) - \{x\}$. Note that twins may or may not be adjacent; if they are we call them \emph{true twins} and if not we call them \emph{false twins}.

Here are some properties of distance-hereditary graphs that we will use (see Howorka \cite{Howorka} and Bandelt and Mulder \cite{BandeltMulder}):

\begin{enumerate}[(DH1)]
\item \label{prop2} Any cycle of length at least 5 in a distance-hereditary graph has two crossing chords (\cite[Theorem 1]{Howorka}).
\item \label{prop3} If $x$ is a vertex in a distance-hereditary graph and $u$ and $v$ are adjacent vertices in $S_i(x)$, then $N_{S_{i-1}(x)}(u) = N_{S_{i-1}(x)}(v)$ (\cite[Theorem 3]{BandeltMulder}).
\item \label{prop4} Any $2$-connected distance-hereditary graph with at least four vertices has two disjoint pairs of twins (\cite[Corollary 1]{BandeltMulder}).
\end{enumerate}

For distinct vertices $x, y$ in a graph $G$, we denote by $Ext(x, y)$ (for \textit{extension of $x,y$}) the set $\{z : [xyz]\}$, by   $I(x, y)$ the set $\{z : [xzy]\}$, by $I[x, y)$ the set $\{x\} \cup \{z : [xzy]\}$ and by $I[x, y]$ the set $\{x,y\} \cup \{z : [xzy]\}$.
With this notation, if $x$ and $y$ are two vertices of a graph $G$, we have
\[
\ov{xy}= Ext(y,x) \cup I[x,y] \cup Ext(x,y).
\]

\section{The main result}

This section is devoted to the proof of Theorem \ref{main}. 
We start with two  lemmas used as tools in the main proof.
A \textit{triangle} is a graph made of three pairwise adjacent vertices.

\begin{lemma}\label{lem:triangle}
Let $G$ be a connected distance-hereditary graph and let $xy$ be an edge of $G$.
Then either the edge $xy$ is contained in a triangle, or $\ov{xy}$ is universal.
\end{lemma}

\begin{proof}
Assume that there exists a vertex $u$ not in $ \ov{xy}$. 
Note that $|d_G(u,x) - d_G(u,y)| \leq 1$ since $xy$ is an edge; but this does not hold with equality for then $u$ would lie in $\ov{xy}$.
Thus there is an integer $i$ such that $x$ and $y$ are both in $S_i(u)$. 
Since $x$ and $y$ both have at least one neighbour in $S_{i-1}(u)$, it follows from Property (DH\ref{prop3}) that $xy$ is contained in a triangle.
\end{proof}

\begin{lemma}\label{lem:xaxb}
Let $G$ be a connected distance-hereditary graph and $x$, $a$, $b$ three distinct vertices of $G$. 
If $\ov{xa} = \ov{xb}$, then either $\ov{ab}$ is a universal line or $[axb]$.
\end{lemma}

\begin{proof}
Assume by way of contradiction that $\ov{ab}$ is not universal and that, without loss of generality, $[xab]$.
Assume that $a\in S_i(x)$, $b \in S_j(x)$  (observe that, since $[xab]$, $i<j$). 
If $ab$ is an edge, then by Lemma \ref{lem:triangle}, it is contained in a triangle $G[\{a,b,c\}]$. 
If $c \in S_j(x)$, then $c \in\ov{xa}$ and $c \notin \ov{xb}$, a contradiction. 
If $c \in S_i(x)$, then $c \notin\ov{xa}$ and $c \in \ov{xb}$, a contradiction. 
Hence we may assume that $ab$ is not an edge.

Observe that, since $G$ is a distance-hereditary graph and $ab$ is not an edge, $G - I(a, b)$ has no path joining $a$ and $b$ (otherwise the distance between $a$ and $b$ in $G - I(a, b)$ is strictly greater than in $G$).
Therefore, as $\ov{ab}$ is not universal, there exists a vertex $u$ in $I(a, b)$ that has a neighbour $v$ not in $I[a,b]$. 
Assume $u \in S_k(x)$ where $i<k<j$.
If $v \in S_{k+1}(x)$,  then $v \in \ov{xa}$ and $v \notin \ov{xb}$, a contradiction.
If $v \in S_{k}(x)$, then by Property (DH\ref{prop3}),  $u$ and $v$ have a common neighbour in $I[a,b] \cap S_{k-1}(x)$ and hence $v \in \ov{xa}$ but $v \notin \ov{xb}$, a contradiction.
If $v \in S_{k-1}(x)$, then $v \notin \ov{xa}$ and $v \in \ov{xb}$, a contradiction.
\end{proof}

We will abuse terminology and say that a set $L \subseteq V(G)$ is a \emph{line of the graph $G$} if it is a line of the metric space induced by $G$. If $V(G)$ is a line of $G$, we will call it a \emph{universal line}.
When we are dealing with lines from more than one graph, we add a superscript and write $\ov{uv}^G$ to specify the line generated by the vertices $u$ and $v$ in the graph $G$.
We now prove our main result.

\vspace{2ex}

\noindent \textit{Proof of Theorem~\ref{main}:}

\vspace{2ex}

Let $G$ be a counterexample with the minimum number of vertices and set $n = |V(G)|$; so $G$ is a connected distance-hereditary graph with at least two vertices and $G$ has at most $n - 1$ lines but no universal line.
Note that by Lemma~\ref{lem:triangle}, every edge of $G$ is contained in a triangle.

\begin{claim} \label{clm:Gis2connected}
$G$ is $2$-connected.
\end{claim}

\begin{proofclaim}
Assume that $G$ has a cutvertex, $x$. 
Let $G_1$ be a component of $G - x$ with $|V(G_1)|$ minimum.
Let $G_2 = G - G_1$ and let $n_2 = |V(G_2)|$. 
By the minimality of $V(G_1)$, $n_2 \ge n/2$.
Let $u$ be a neighbour of $x$ in $G_1$. 
Since $G$ has no universal line, by Lemma~\ref{lem:triangle} the edge $xu$ is contained in a triangle; call the third vertex of this triangle $v$. Note that $v \in V(G_1)$.

Let $a$ and $b$ be two vertices in $G_2$. 
It is clear that $[aub]$ cannot hold.
Since $G$ has no universal line, it follows from Lemma~\ref{lem:xaxb} that we have $\ov{ua} \neq \ov{ub}$.  
Similarly, $\ov{va} \neq \ov{vb}$.
Moreover, for any $c \in G_2$, $\ov{uc}$ does not contain $v$ and $\ov{vc}$ does not contain $u$, because $d_G(u, v) = 1$ and $d_G(u, c) = d_G(v, c)$. Hence $\ov{uc} \neq \ov{vc}$.
Therefore, $G$ has at least $2n_2 \ge n$ lines, a contradiction.
\end{proofclaim}

\begin{claim} \label{clm:twin2connected}
If $\{x, y\}$ is a pair of twins in $G$, then $G - y$ is $2$-connected.
\end{claim}

\begin{proofclaim}
Assume that $G - y$ has a cutvertex, $t$.
If $t \neq x$, then $t$ is also a cutvertex of $G$, contradicting Claim~\ref{clm:Gis2connected}. 
So $x$ is a cutvertex of $G - y$ and thus $\{x, y\}$ is a $2$-vertex cutset of $G$. 
Let $G_1$ be a component of $G - \{x, y\}$ with $|V(G_1)|$ minimum, and let $G_2 = G - G_1$. 
By the minimality of $G_1$, $n_2 = |V(G_2)| \ge (n-2)/2 + 2 = (n+2) / 2$.

Suppose that $|V(G_1)| = 1$. Let $u$ be the unique vertex of $G_1$. Since every edge of $G$ is contained in a triangle, $x$ and $y$ are adjacent. 
Observe that for any distinct $a, b \in V(G_2)$, $[aub]$ does not hold, so $\ov{ua} \neq \ov{ub}$. 
This gives us $n-1$ distinct lines. 
Moreover, the line $\ov{xy}$ does not contain $u$ so it is distinct from $\ov{ua}$ for all $a \in V(G_2)$. 
Hence $G$ has at least $n$ distinct lines, a contradiction.
We may therefore assume that $|V(G_1)| \geq 2$.
Since $G$ is $2$-connected, $x$ and $y$ have at least two neighbours in $G_1$; let $u$ and $v$ be two of these neighbours.
If $x$ and $y$ are false twins then, since every edge is contained in a triangle, we can (and we do) choose $u$ and $v$ adjacent.

Assume first that $x$ and $y$ are true twins, that is, $xy$ is an edge.
Then it is easy to see that for any distinct vertices $a, b \in V(G_2)$, neither $[aub]$ nor $[avb]$ can hold.
Hence, by Lemma~\ref{lem:xaxb}, $\ov{ua} \neq \ov{ub}$ and $\ov{va} \neq \ov{vb}$.
Moreover, for any vertex $c \in V(G_2) - \{x, y\}$, $v \notin \ov{uc}$ and $u \notin \ov{vc}$.
Hence, we have the following set of distinct lines: $\{\ov{uc} : c \in V(G_2)\} \cup \{\ov{vc} : c \in V(G_2) - \{x, y\}\}$.
This gives at least $n_2 + n_2 - 2 \geq n$ distinct lines, a contradiction.

So we may now assume that $x$ and $y$ are false twins, so $xy$ is not an edge and $uv$ is an edge.
For any distinct $a, b \in V(G_2) - \{y\}$, we have $\ov{ua} \neq \ov{ub}$ and $\ov{va} \neq \ov{vb}$ by Lemma~\ref{lem:xaxb}. 
Moreover, for any $c \in V(G_2)$,  since $v \notin \ov{uc}$ and $u \notin \ov{vc}$, we have $\ov{uc} \neq \ov{vc}$.
Hence we have the following set of distinct lines: $\{\ov{uc} : c \in V(G_2) - \{y\}\} \cup \{\ov{vc} : c \in V(G_2) - \{y\}\}$. This gives  $2(n_2 - 1) \geq n$ distinct lines, a contradiction.
\end{proofclaim}

\begin{claim}\label{ulinG-y}
If $\{x,y\}$ is a pair of twins, then $G-y$ has a universal line.
\end{claim}

\begin{proofclaim}
By Claim~\ref{clm:twin2connected}, $G - \{x, y\}$ is connected.
By the minimality of $G$, $G-y$ either has at least $n-1$ lines or has a universal line. 
Assume by way of contradiction that $G - y$ has at least $n - 1$ lines.
By definition of distance-hereditary graphs and because $x$ and $y$ are twins, it is easy to see that, for any two vertices $s,t$ in $ V(G) - \{y\}$, the following holds:
\begin{itemize}
\item if $x \in \{s, t\}$, then either $\ov{st}^G = \ov{st}^{G - y}$ or $\ov{st}^G = \ov{st}^{G - y} \cup \{y\}$,
\item if $x \notin \{s, t\}$ and $x \in \ov{st}^{G - y}$, then $\ov{st}^G = \ov{st}^{G - y} \cup \{y\}$,
\item if $x \notin \{s, t\}$ and $x \notin \ov {st}^{G-y}$, then $\ov{st}^G = \ov{st}^{G - y}$.
\end{itemize}

So the set $\{\ov{st}^G : s, t \in V(G) - \{y\}\}$ contains at least $n - 1$ distinct lines of $G$, and each of them either contains both $x$ and $y$, or does not contain $y$.
Therefore, no line of $G$ contains $y$ but not $x$ because $G$ has at most $n-1$ lines.

If $x$ and $y$ are adjacent, then for every $t$ in $V(G) - \{x, y\}$,   $\ov{yt}^G$ contains  $y$ but not $x$, a contradiction. 
If there is a vertex $t$ in $V(G) - \{x, y\}$ that is not adjacent to $x$, then $\ov{yt}^G$  contains $y$ but not $x$, a contradiction. 
Thus $x$ and $y$ are not adjacent and every other vertex of $G$ is adjacent to both $x$ and $y$. 
Then $\ov{xy}^G$ is a universal line of $G$, a contradiction.
\end{proofclaim}

\begin{claim} \label{last}
 If $\{x, y\}$ is a pair of twins in $G$, then there is a vertex $z$ in $V(G) - \{x, y\}$ that is not adjacent to $x$ and such that $V(G) = \{y\} \cup I[x, z]$.
\end{claim}

\begin{proofclaim}
Assume there exists two vertices $s$ and $t$ in $V(G)-\{x,y\}$ such that $\ov{st}^{G-y}$ is a universal line of $G-y$.
Since   $x \in \ov{st}^{G - y}$,    $y \in \ov{st}^G$ and thus $\ov{st}^G$ is a universal line of $G$, a contradiction. 
Since $G-y$ has a universal line by Claim~\ref{ulinG-y}, it follows that there exists a vertex $z \in V(G)-\{x,y\}$ such that $\ov{xz}^{G-y}$ is a universal line of $G-y$.
Thus $\ov{xz}=Ext(z,x) \cup I[x,z] \cup Ext(x,z)=V(G)-y$ so, to prove the claim, it suffices to show that $x$ and $z$ are not adjacent and $Ext(x,z)=Ext(z,x)=\emptyset$.

Suppose that $x$ and $z$ are adjacent. Then $y$ and $z$ are also adjacent, and as $y \notin \ov{xz}^G$, $xy$ is an edge. Also, $I(x, z) = \emptyset$, and thus $ Ext(z, x) \cup Ext(x, z)=V(G) - \{x, y, z\} $.
For any distinct vertices $a_1, a_2 \in Ext(z, x)$, $[a_1za_2]$ does not hold so, by Lemma~\ref{lem:xaxb}, $\ov{za_1} \neq \ov{za_2}$.
Similarly, for any distinct vertices $b_1, b_2 \in Ext(x, z) \cup \{z\}$, we have $\ov{xb_1} \neq \ov{xb_2}$.
Moreover, for any $a \in Ext(z, x)$ and $b \in Ext(x, z) \cup \{z\}$ we have $\ov{za} \neq \ov{xb}$ because $y \in \ov{za}$ but $y \not\in \ov{xb}$. 
Thus $\{ \ov{za} : a \in Ext(z, x)\} \cup \{\ov{xb} : b \in Ext(x, z) \cup \{z\}\}$ is a set of $n-2$ distinct lines of $G$.
We then observe that $\ov{xy}$ and $\ov{yz}$ are two lines that are distinct from all of these, because $\ov{xy}$ does not contain $z$ and $\ov{yz}$ does not contain $x$. Then $G$ has at least $n$ lines. So $x$ and $z$ are not adjacent.

There is no edge with one endpoint in $I(x, z) \cup \{x\}$ and the other one in $Ext(x, z)$ and, similarly, there is  no edge with one endpoint in $I(x, z) \cup \{z\}$ and the other one in  $Ext(z, x)$. 
Suppose there is an edge $ab$ with $a \in Ext(z, x)$ and $b \in Ext(x, z)$. 
We have $[zxa]$ and $[xzb]$ so 
\[ d_G(a, z) = d_G(a, x) + d_G(x, z) \leq 1 + d_G(b, z) \]
 and 
\[ d_G(b, x) = d_G(b, z) + d_G(z, x) \leq 1 + d_G(a, x). \]
 Hence  $d_G(a, x) + d_G(b, z) + 2d_G(x, z) \leq 2 + d_G(a, x)+d_G(b, z)$, implying that $d_G(x, z) = 1$, a contradiction.
So there is no edge with one endpoint in $Ext(z, x)$ and the other one in  $Ext(x, z)$.

So $Ext(x, z)$ is empty, for otherwise $z$ would be a cutvertex of $G$, and   $Ext(z, x)$ is empty, otherwise $x$ would be a cutvertex of $G - y$.
\end{proofclaim}

We have now proved enough claims to finish the proof.
It is easy to  check that $n \ge 4$.
By Claim~\ref{clm:Gis2connected}, $G$ is 2-connected, so by Property (DH\ref{prop4}), $G$ has two disjoint pairs of twins $\{x, y\}$ and $\{u, v\}$. 
By Claim~\ref{last} and because $x$ and $y$ are twins, there is a vertex $z \in V(G) - \{x, y\}$ that is not adjacent to $x$ such that $V(G) = \{y\} \cup I[x, z]$. 
Similarly, there is a vertex $w \in V(G) - \{u, v\}$ that is not adjacent to $u$ such that $V(G) = \{v\} \cup I[u, w]$.

Observe that $u \neq z$, otherwise we would have $[xvu]$, contradicting the fact that $u$ and $v$ are twins (similarly $v \neq z$). So we have $[xuz]$.
Similarly, we have $x \neq w$ and $[uxw]$ (also $y \neq w$).

Moreover, since $[xuz]$ and $[uxw]$, we have $w \neq z$. 
Hence the six elements of $\{x, y, z, u, v, w\}$ are pairwise distinct. 
Thus we have $w,u \in I(x, z)$ and  $z,x \in I(u, w)$  i.e.~$[xwz]$,  $[xuz]$, $[uzw]$, and $[uxw]$.
We are now going to show that these four properties cannot all hold together.

Let $P_{xu}$, $P_{uz}$, $P_{zw}$ and $P_{wx}$ be shortest paths joining the pairs $\{x, u\}, \{u, z\}, \{z, w\}$ and $\{w, x\}$, respectively. 
Observe that, since $[xuz]$, $P_{xu} \cup P_{uz}$ is a shortest path from $x$ to $z$ (going through $u$), and similarly $P_{uz} \cup P_{zw}$, $P_{zw} \cup P_{wx}$ and $P_{wx} \cup P_{xu}$ are shortest paths from respectively $u$ to $w$ (going through $z$), $z$ to $x$ (going through $w$) and $w$ to $u$ (going through $x$).

We claim that $P_{xu}$ and $P_{zw}$ are disjoint; indeed, if they are not, then there is a shortest path from $z$ to $w$ going through $u$, i.e.~$[zuw]$ contradicting the fact that $[uzw]$.
Similarly, $P_{uz}$ and $P_{wx}$ are disjoint. 
Hence, the unions of the paths $P_{xu}, P_{uz}, P_{zw}$ and $P_{wx}$ forms a cycle, $C$.

If $|E(C)| = 4$, then $d_G(x, z) = 2$ and the fact that $V(G) - \{y\} \subseteq I[x, z]$ means that $G - y$ has diameter two. As $x$ and $y$ are twins in $G$, $G$ also has diameter two which contradicts the result cited in the introduction stating that any $2$-metric space on $n \ge 2$ points either has at least $n$ distinct lines, or has a universal line \cite{Chvatal}. 
So $|E(C)| \geq 5$.

Now by Property (DH\ref{prop2}) applied to the distance-hereditary graph $G - \{y, v\}$, the cycle $C$ has two crossing chords, $e$ and $f$. 
Let $e_1$ and $e_2$ be the extremities of $e$ and $f_1$ and $f_2$ the extremities of $f$.
Note that no chord exists with both edges in $P_{xu} \cup P_{wx}$, both edges in $P_{uz} \cup P_{zw}$, both edges in $P_{xu} \cup P_{uz}$, or both edges in $P_{zw} \cup P_{wx}$.
Hence each of $e$ and $f$ either joins the interiors of $P_{uz}$ and $P_{wx}$ or those of $P_{xu}$ and $P_{zw}$. 

First, we suppose that $e_1, e_2, f_1,$ and $f_2$ respectively lie in the interiors of the paths $P_{uz}$, $P_{wx}$, $P_{xu}$ and $P_{zw}$.
Then we have
\[ d_G(x, e_2) + 1 + d_G(e_1, z) \geq d_G(x,z) = d_G(x, e_2) + d_G(e_2, w) + d_G(w, z) \]
so that $d_G(u, z) \geq 1 + d_G(e_1, z) > d_G(w, z)$.
We also have
\[ d_G(x, f_1) + 1 + d_G(f_2, z) \geq d_G(x,z)=d_G(x, f_1) + d_G(f_1, u) + d_G(u, z) \]
so that $d_G(w, z) \geq 1 + d_G(f_2, z) > d_G(u, z)$, a contradiction.

Thus we may assume by symmetry that $e_1$ and $f_1$ both lie in the interior of $P_{uz}$ and $e_2$ and $f_2$ both lie in the interior of $P_{wx}$. We may assume by symmetry that $[e_1f_1z]$ and $[e_2f_2x]$.
Now we have 
\[ d_G(x, f_2) + 1 + d_G(f_1, z) \geq d_G(x,z)= d_G(x, f_2) + d_G(f_2, w) + d_G(w, z) \]
so that $d_G(e_1, z) \geq 1 + d_G(f_1, z) > d_G(f_2, w)$.
We also have
\[ d_G(u, e_1) + 1 + d_G(e_2, w) \geq d_G(u,w)=d_G(u, e_1) + d_G(e_1, z) + d_G(z, w) \]
so that $d_G(f_2, w) \geq 1 + d_G(e_2, w) > d_G(e_1, z)$, a contradiction. \hfill{$\Box$}

\vspace{2ex}

Although our proof finds $n$ lines in the metric space induced by an $n$-vertex distance-hereditary graph, it seems likely that this is not the best possible lower bound. In fact, Va\v{s}ek Chv\'atal has asked whether the following conjecture is true for all graphs:

\begin{conjecture}
The metric space induced by any connected graph on $n$ vertices ($n \geq 2$) either has a universal line or has $\Omega(n^{4/3})$ lines.
\end{conjecture}

An example of an infinite family of graphs in which each graph $G$ has $\Omega(|V(G)|^{4/3})$ lines is the complete multipartite graphs $G$ whose vertices can be partitioned into $|V(G)|^{2/3}$ independent sets of size $|V(G)|^{1/3}$.
These graphs are in fact distance-hereditary.
More generally, it was proved by Chiniforooshan and Chv\'atal \cite{CC2} that any $2$-metric space on $n$ points has $\Omega(n^{4/3})$ lines (which implies that this conjecture is true for the class of graphs of diameter two).
Recently, the two authors and Supko \cite{newBound} have proved that every metric space induced by a connected graph on $n$ vertices either has $\Omega(\sqrt{n})$ distinct lines or has a universal line.

\section*{Acknowledgements}

We thank Va\v{s}ek Chv\'atal for introducing us to this subject and for his helpful comments on our manuscript, and him, Xiaomin Chen, and Cathryn Supko for many useful discussions.
This research was undertaken, in part, thanks to funding from the Canada Research Chairs program and from the Natural Sciences and Engineering Research Council of Canada.

\end{document}